\documentclass[reqno, 11pt]{amsart}

\usepackage{amsmath}
\usepackage{amsthm}
\usepackage{graphicx}
\usepackage{verbatim}
\usepackage{epsfig}
\usepackage{hyperref}
\usepackage{float}
\usepackage{color}
\usepackage{fullpage}
\usepackage{amsfonts}
\usepackage{amscd}
\usepackage{amssymb}
\usepackage{graphics}
\usepackage{amsmath}
\usepackage[english]{babel}
\usepackage{bbm}
\usepackage{yfonts}
\usepackage{mathrsfs}
\usepackage{color}
\usepackage{calc,amsfonts,amsthm,amscd,epsfig,psfrag,amsmath,amssymb,enumerate,graphicx}
\DeclareFontFamily{OML}{rsfs}{\skewchar\font'177}
\DeclareFontShape{OML}{rsfs}{m}{n}{ <5> <6> rsfs5 <7> <8> <9>
rsfs7 <10> <10.95> <12> <14.4> <17.28> <20.74> <24.88> rsfs10 }{}
\DeclareMathAlphabet{\mathfs}{OML}{rsfs}{m}{n}

\newcommand{\bP}{{\bf P}}

\newcommand{\cI}{\ensuremath{\mathcal I}}

\newcommand{\cN}{\ensuremath{\mathcal N}}

\newcommand{\cT}{\ensuremath{\mathcal T}}

\newcommand{\cW}{\ensuremath{\mathcal W}}

\DeclareMathOperator{\tr}{\mathrm{Tr}}
\DeclareMathOperator{\cp}{\mathrm{cap}}

\newcommand{\bbR}{{\ensuremath{\mathbb R}} }

\newcommand{\bbT}{{\ensuremath{\mathbb T}} }

\newcommand{\bbZ}{{\ensuremath{\mathbb Z}} }

\newcommand{\gep}{\varepsilon}

\newcommand{\go}{\omega}

\newcommand{\dd}{\mathrm{d}}

\newcommand{\be}{\mathbf{e}}

\newcommand{\BP}{{\mathbb{P}}}

\newcommand{\BZ}{{\mathbb{Z}}}

\newcommand{\CE}{{\mathcal{E}}}

\newcommand{\CT}{{\mathcal{T}}}

\newcommand{\ind}{{\mathbbm{1}}}

\newcommand{\bae}{\begin{equation}\begin{aligned}}
\newcommand{\eae}{\end{aligned}\end{equation}}

\newcommand{\Z}{\mathbb{Z}}

\newtheorem{thm}{Theorem}[section]
\newtheorem{prop}[thm]{Proposition}
\newtheorem{lem}[thm]{Lemma}

\newtheorem{rem}[thm]{Remark}
\newtheorem{definition}{Definition}[section]

\newtheorem{rmk}{Remark}
\begin{document}
\numberwithin{equation}{section} \numberwithin{figure}{section}
\title[Connectivity for Random Interlacements]{On the easiest way to connect $k$ points in the Random Interlacements process}

\author[H. Lacoin]{Hubert Lacoin}
\address{H. Lacoin,
CEREMADE - UMR CNRS 7534 - Universit\'e Paris Dauphine,
Place du Mar\'echal de Lattre de Tassigny, 75775 CEDEX-16 Paris, France. \newline
e--mail: {\tt lacoin@ceremade.dauphine.fr}}
\author[J. Tykesson]{Johan Tykesson}
\address{J. Tykesson,
Department of Mathematics, Uppsala University, PO Box 480, SE-751 06
Uppsala, Sweden. \newline e--mail: {\tt johan.tykesson@math.uu.se}}

\begin{abstract}
We consider the random interlacements process with intensity $u$  on ${\mathbb Z}^d$, $d\ge 5$ (call it $\mathcal I^u$), built
from a Poisson point process on the space of doubly infinite nearest neighbor trajectories on ${\mathbb Z}^d$.
For $k\ge 3$ we want to determine the minimal number of trajectories from the point process that is needed to link together $k$ points in $\mathcal I^u$.
Let $$n(k,d):=
\lceil \frac d 2 (k-1) \rceil - (k-2).$$ We prove that almost surely given any $k$ points $x_1,...,x_k\in \mathcal I^u$, there is  a sequence of
$n(k,d)$ trajectories $\gamma^1,...,\gamma^{n(k,d)}$ from the underlying Poisson point process such that the union of their traces
$\bigcup_{i=1}^{n(k,d)}\tr(\gamma^{i})$ is a connected set containing $x_1,\dots,x_k$.
Moreover we show that this result is sharp, \textit{i.e.}\ that a.s.\ one can find $x_1,...,x_k\in \mathcal I^u$ that cannot be linked together
by $n(k,d)-1$ trajectories.
\end{abstract}

\maketitle

\section{Introduction}

The random interlacement set is the trace left by a
Poisson point process on the space of doubly infinite nearest neighbor trajectories modulo time shift
on ${\mathbb Z}^d$. This Poisson point process
is governed by the intensity measure $u\nu$ where $u>0$ and $\nu$ is a measure on the space of
doubly infinite trajectories which was constructed by Sznitman in \cite{sznitmanvacant}, see~\eqref{e:nudef} below.
This measure essentially makes the trajectories in the Poisson point process look like double sided simple random walk paths.
The interlacement set is a site percolation model that exhibits polynomially decaying infinite-range dependence which sometimes complicates analysis.

One of the motivations for introducing the random interlacements model was to use it as a tool for the study of the behavior of simple random walks on large but finite graphs. For instance, random interlacements describe the local picture left by the trace of a simple random walk on a discrete torus or a discrete cylinder, see \cite{W08} and \cite{Szn09b} respectively.
Recent works that also have used random interlacements to obtain results about simple random walks on large graphs are for example \cite{Szn09d}, \cite{Szn09c}, \cite{TeixWind10}, \cite{Belius1} and \cite{Belius2}.

 It is known that the interlacement set is always a connected set, see Corollary (2.3) in \cite{sznitmanvacant}. Recently, in
\cite{rath2010connectivity} and \cite{PT2011} a stronger result was shown: given any two points $x$ and $y$ in the
interlacement set, one can find a path between $x$ and $y$ using the trace of at most $\lceil d/2 \rceil$
trajectories. The proofs in \cite{rath2010connectivity} and \cite{PT2011} are very different; in \cite{PT2011}
 the concept of stochastic dimension from \cite{benjamini2004geometry} is used, while in
\cite{rath2010connectivity} the approach of the problem is based on estimating capacities of random sets constructed using random walks.

\medskip

The result we present in this paper completes these works, giving a full picture of how a finite number of points are connected
together within the interlacement set. Fix $k\ge 2$, and $d\ge 5$,
given a realization $\mathcal I^u$ of the random interlacement of intensity $u$ constructed from the Poisson point process $\go_u$ on the space of
double trajectories (see the next section for formal definition), a.s.\ for any sequence of points
$x_1,...,x_k\in \mathcal I^u$, there is a sequence of $n(k,d)$ trajectories
$\gamma^1,\dots,\gamma^{n(k,d)}\in \go_u$ such that
\begin{itemize}
 \item [(a)] $\bigcup_{i=1}^{n(k,d)} \tr(\gamma^{n(k.d)})$ is a connected set (where $\tr$ denote the trace or image of a doubly infinite trajectory $\gamma: \bbZ\to \bbZ^d$),
 \item [(b)] $x_j\in \bigcup_{i=1}^{n(k,d)} \tr(\gamma^{n(k.d)}) \quad \forall j\in [1,k]$.
\end{itemize}

\medskip

 In addition, this result is sharp: of course the $n(k,d)$ trajectories are not always needed to link the $k$ points
(e.g.\ $x_1,\dots,x_k$ might all lie on the trace of a common trajectory) but
with probability one, there exist $y_1,...,y_k\in \mathcal I^u$ such that but there is no sequence of $n(k,d)-1$ trajectories such that
satisfies the two conditions $(a)$ and $(b)$ above.

\medskip

The main results from \cite{rath2010connectivity} and \cite{PT2011} corresponds to the case $k=2$.
The proof of the upper bound for $n(k,d)$ pushes the techniques developed in \cite{rath2010connectivity} further,
 while the proof of the lower bound uses a more novel approach based on diagrammatic sums.

\medskip

In the next section we give a rigorous definition of the random interlacement process and
state our result in full detail.


\section{Notation and results}\label{s.notationresults}

\subsection{Definition and construction of random interlacements}\label{s.randominterlace}
We consider the trajectory spaces $W$ and $W_+$ of doubly infinite and infinite transient nearest neighbor trajectories in $\Z^d$ (and $\cW, \cW_+$ the usual sigma algebras associated to them):
$$W:=\{\gamma\,:\,\Z\to \Z^d;\,|\gamma(n)-\gamma(n+1)|=1,\,\forall n\in
\Z;\,|\{n;\,\gamma(n)=y\}|<\infty,\,\forall y\in{\mathbb Z}^d\},$$
$$W_+:=\{\gamma\,:{\mathbb N}\to \Z^d;\,|\gamma(n)-\gamma(n+1)|=1,\,\forall n\in
\Z;\,|\{n;\,\gamma(n)=y\}|<\infty,\,\forall y\in{\mathbb Z}^d\},$$
where we use the convention that ${\mathbb N}$ includes $0$.
For $\gamma\in W$, we define the trace of $\gamma$, $\tr(\gamma)=\{\gamma(n), \ n \in \bbZ\}$. For trajectories $\gamma,\gamma'\in W$, we write $\gamma\sim \gamma'$ if for some $k\in \Z$ we have $\gamma(\cdot)=\gamma'(\cdot+k)$. The space of trajectories in $W$ modulo time shift will be denoted by $W^*$ and is defined as follows:
$$W^*:=W/\sim.$$
As the trace is invariant modulo time-shift we can naturally extend the notion of trace to $W^*$.

For $K\subset{\mathbb Z}^d$ and $\gamma\in W_+$, we let $H_K(\gamma)$,
$\tilde{H}_K(\gamma)$ and $T_K(\gamma)$ denote the entrance time, hitting time and exit time of $K$ by $\gamma$:
\begin{equation}\label{e.entrance}
H_K(\gamma):=\inf\{n\ge 0\,: \gamma(n) \in K\},
\end{equation}
\begin{equation}\label{hitting}
\tilde{H}_K(\gamma):=\inf\{n\ge 1\,:\, \gamma(n)\in K\},
\end{equation}
\begin{equation}\label{e.exit}
T_K(\gamma):=\inf\{n\ge 0\,:\, \gamma(n) \notin K\}.
\end{equation}

For $x\in {\mathbb Z}^d$, set $H_x:= H_{\{x\}}$. Let $P_x$ be the law on $W_+$ which corresponds to a simple
(i.e.\ nearest-neighbor symmetric) random walk on $\bbZ^d$ started at $x$. For
$K\subset \BZ^d$, let $P_x^K$ be the law of simple random walk started at $x$
conditioned on the event that the walk does not hit $K$:$$P_x^K[\cdot]:=P_x[\cdot|\tilde{H}_K=\infty].$$ For a finite $K\subset {\mathbb Z}^d$, we define the equilibrium measure
\bae e_K(x):=\left\{\begin{array}{ll}
P_x[\tilde{H}_K=\infty],\quad &x\in K\\
0,\quad &x\notin K.\end{array}\right. \eae
The capacity of
a finite set $K\subset \BZ^d$ is defined as
\begin{equation}\label{capacity}
\cp(K):=\sum_{x\in \BZ^d} e_K(x).
\end{equation}
 and the normalized equilibrium measure of
$K$ is given by
\begin{align}
\tilde{e}_K(\cdot):=e_K(\cdot)/\cp(K).
\end{align}
For $x,y\in {\mathbb Z}^d$ we let $|x-y|:=\|x-y\|_1$
denote the $l_1$ distance (which corresponds to the graph distance on $\bbZ^d$) between $x$ and $y$.
The following bounds of hitting-probabilities are well-known, see Theorem 4.3.1 in \cite{lawler2010random}. For any $x,y\in {\mathbb Z}^d$ with $x\neq y$,
\begin{equation}\label{e.hitbounds}
c |x-y|^{-(d-2)}\le P_x[\tilde{H}_y<\infty]\le c' |x-y|^{-(d-2)}.
\end{equation}
 We are now ready to introduce a Poisson point process on $W^*\times{\mathbb R}_+$. 
For $K\subset{\mathbb Z}^d$, let $$W_K:=\{\gamma\in W\,:\,\gamma({\mathbb Z})\cap K\neq\emptyset\}.$$
Let $\pi^*$ be the projection from $W$ to $W^*$ and let $W_K^{*}:=\pi^*(W_K)$ be the set of trajectories in $W^*$ that enter $K$. We denote by $Q_K$ the finite measure on $W_K$ such that for $A,B\in {\mathcal
W}_+$ and $x\in {\mathbb Z}^d$,
\begin{equation}\label{e:Qdef}Q_K[(X_{-n})_{n\ge 0}\in A,X_0=x,(X_n)_{n\ge 0}\in B]=P_x^K[A]e_K(x)P_x[B].\end{equation} We let the measure $\nu$ be the unique $\sigma$-finite measure such that
\begin{equation}\label{e:nudef}
\ind_{W^*_K}\nu=\pi^*\circ Q_K,\text{ for all finite } K\subset{\mathbb Z}^d.
\end{equation}
Sznitman proved the existence and uniqueness of $\nu$ in Theorem 1.1 of \cite{sznitmanvacant}. We introduce the space of locally finite point measures
in $W^{*}\times {\mathbb R}_+$:
\bae
\Omega:=\bigg\{&\omega=\sum_{i=1}^{\infty}\delta_{(\gamma_i,u_i)};\,\gamma_i\in W^*, u_i>0,\,\\
&\omega(W_K^*\times [0,u])<\infty,\mbox{ for every finite }K\subset \Z^d\text{ and }u>0\bigg\},
\eae
as well as the space of locally finite point measures on $W^*$:
\begin{equation}\tilde{\Omega}:=\left\{\sigma=\sum_{i=1}^{\infty}
\delta_{\gamma_i};\,\gamma_i\in W^*,\, \sigma(W_K^*)<\infty,\,\mbox{ for
every finite }K\subset \Z^d\right\}.
\end{equation} For $0\le u'\le u$  the map $\omega_{u',u}$ from $\Omega$ into
$\tilde{\Omega}$ is defined as
\begin{align}\label{e.omegaudef}\omega_{u',u}:=\sum_{i=1}^\infty \delta_{\gamma_i}{\ind}\{u'< u_i\le u\},\text{ for }\omega=\sum_{i=1}^{\infty}
\delta_{(\gamma_i,u_i)}\in \Omega.\end{align} If $u'=0$, we use the short-hand notation
$\omega_u$. For convenience reasons we often improperly consider $\go_u$ as a set of trajectories instead of a point measure.

\medskip

 On $\Omega$ we consider
${\mathbb P}$, the law of a Poisson point process with intensity
measure $\nu(d\gamma)dx$ (see Equation (1.42) in \cite{sznitmanvacant} for a characterization of ${\mathbb P}$). It is easy to see that under ${\mathbb P}$,
the point process $\omega_{u,u'}$ is a Poisson point process on
$\tilde{\Omega}$ with intensity measure $(u-u') \nu(dw^*)$. Given
$\sigma\in \tilde{\Omega}$, the set of points in ${\mathbb Z}^d$ that is visited by at least one trajectory in $\sigma$ is denoted by
\begin{equation}\label{e.nicenotation}
{\mathcal
I}(\sigma):=\bigcup_{\gamma\in\sigma}\tr(\gamma).
\end{equation}
For $0\le u'\le u$, we define the \emph{random interlacement set} between intensities
$u'$ and $u$ as
\begin{equation}\label{e.ridef}
{\mathcal I}^{u',u}:={\mathcal I}(\omega_{u',u}).
\end{equation}
 In case $u'=0$, we use the short-hand notation ${\mathcal I}^u$.
For a point process $\sigma$ on $\Omega$ or $\tilde{\Omega}$ we let $\sigma|_{A}$ denote the restriction of $\sigma$ to $A\subset W^*$.

We conclude this section by stating the convention for the use of constants throughout the paper: The letters $c,c',C,C'$ etc.
denote finite positive constants which are allowed to depend only on the dimension $d$ and the intensity $u$. Their values might change
from line to line. Numbered constants $c_i$ are finite positive, and supposed to be the same inside a certain neighborhood (for example a proof).
They are defined where they first appear. Dependence of additional quantities will be indicated, for example $c_{\delta}$ denotes a constant
that might depend on $d$, $u$ and $\delta$. Moreover when needed we will identify trajectory $\gamma \in W^*$, with a canonical element of its
equivalence class $(\gamma_n)_{n\ge 0}$.

\subsection{Main result}\label{s.mainresults}

We say that the sequence of trajectories $(\gamma^i)_{i=1}^n$ connects the sequence of points $(x_i)_{i=1}^k$ if
the union of their traces (or images) includes a connected subset that contains $x_1 , . . . , x_k$. We
say that $(\gamma^i)_{i=1}^n$ connects strictly $(x_i)_{i=1}^k$ if it connects it and there is no strict subsequence of $(\gamma^i)_{i=1}^n$ that does. Note
that if a sequence of trajectories connects points, one can extract from it a subsequence
that connects them strictly.

\begin{thm}\label{theresult}
For every $k\ge 2$, for every $u>0$, and for ${\mathbb P}-$almost every realization of the Poisson process
$\go_u$, the two following properties are satisfied:

\begin{itemize}
\item[(i)] Given a sequence of $k$ points $(x_i)_{i=1}^k$ in $({\mathcal I}^u)^k$, it is possible to find a sequence
$(\gamma^i)_{i=1}^{n(k,d)}$ in $(\go_u)^{n(k,d)}$ that connects it.

\item[(ii)] It is possible to find $(x_i)_{i=1}^k$ in $({\mathcal I}^u)^k$ such that there exists no sequence
$(\gamma^i)_{i=1}^{n(k,d)-1}\in(\go_u)^{n(k,d)}$ that connects it.
\end{itemize}
\end{thm}

\begin{rem}\rm
The result is restricted to $d\ge 5$ but this is not in fact a true restriction. Indeed if $d=3$ or $4$ the trace of each trajectory in $\go_u$
intersect the trace of all the others, so that Theorem \ref{theresult} trivially holds with $n(k,3)=n(k,4)=k$.
 \end{rem}

The proofs of $(i)$ and $(ii)$ are quite independent and are found in Section~\ref{s.upper} and
Section~\ref{s.lower} respectively.
In what follows we  say that a sequence of points $(x_i)_{i=1}^k$ is $n$-connected (in $(\mathcal I^u)$) if $(i)$ occurs with  $n(k,d)$ replaced by $n$.

\section{Proof of $(i)$ of Theorem \ref{theresult}}\label{s.upper}

As will be seen later in this section, in order to prove that $n(k,d)$ trajectories
are sufficient to connect $k$ points, it is essentially sufficient to prove this in the case $k=2$ and $k=3$.
The case $k=2$ having been proved in \cite{rath2010connectivity} and \cite{PT2011}, we can focus on the case $k=3$.

\medskip

The first step is to reformulate the result.

\begin{prop}\label{p.threecase}
Let $d\ge 5$ and suppose $x_1$, $x_2$, $x_3$ in $\bbZ^d$.
Let $X^1$, $X^2$, $X^3$ be three independent simple random walks on $\bbZ^d$ with starting points $x_1$, $x_2$, $x_3$ respectively.
Consider also a random-interlacement process $\go_u$ independent of the $X^i$s.

\medskip
For any choice of the $x_i$, for every $u>0$,
almost surely one can find $d-4$ trajectories $(\gamma^i)_{i=1}^{d-4}$ in $(\go_u)^{d-4}$ such that
the union of the traces of the $\gamma^i$s forms a connected subset that intersects the traces of $X^1$, $X^2$ and $X^3$.
\end{prop}

We also need a similar result for the case of two trajectories, which is proved in \cite{rath2010connectivity} (with a different formulation).

\begin{prop}\label{p.twocase}
Let $d\ge 5$ and suppose $x_1$, $x_2$ in $\bbZ^d$.
Let $X^1$, $X^2$ be two independent simple random walks on $\bbZ^d$ with starting points $x_1$, $x_2$ respectively.
Consider also a random-interlacement process $\go_u$ which is independent of the walks $X^1$ and $X^2$.

\medskip
For every choice of $x_1$, $x_2$, for every $u>0$,
almost surely one can find $\lceil d/2 \rceil -2$ trajectories $(\gamma^i)_{i=1}^{\lceil d/2 \rceil -2}$ in $(\go_u)^{\lceil d/2 \rceil -2}$ such that
the union of the traces of the $\gamma^i$s forms a connected subset that intersects the traces of $X^1$ and $X^2$.
\end{prop}

\begin{rmk} \rm
Notice that when $d$ is even Proposition \ref{p.threecase} can easily be deduced from Proposition \ref{p.twocase}.
Hence in what follows, we will only care about the case $d$ odd.

\end{rmk}

\begin{proof}[Proof of Theorem \ref{theresult} (i) from Proposition \ref{p.threecase} and \ref{p.twocase}]

First consider the case where $k=2p+1$ is odd. Let $x_1, \dots, x_k$ in $\bbZ^d$.
We say that the sequence of points $(x_1,\dots,x_k)$ is \textsl{well behaved} for $\go_u$, and we will write $WB$, if each point of the sequence belongs
to the interlacement set
and if there exists a sequence $0<t_1<\dots<t_k\le u$ such that for all $i$ there exists $(\gamma^i,t_i)\in \go$  with $x_i\in \gamma^i$.

An equivalent formulation of $(i)$ from Theorem \ref{theresult} is

\begin{equation}\label{trewia}
\begin{split}
\mbox{For all }& k\mbox{ and for all } (x_i)_{i=1}^k\in (\bbZ^d)^k\mbox{ we have that}\\
&\bP\left[ \exists (\gamma^i)_{i=1}^{n(k,d)}, (\gamma^i)_{i=1}^{n(k,d)} \text{ connects } (x_i)_{i=1}^k
\ | \ (x_i)_{i=1}^k\ WB \right]=1,
\end{split}
\end{equation}
or alternatively
\begin{equation}
\mbox{For all }k,\,\bP\left[ \forall (x_i)_{i=1}^k, (x_i)_{i=1}^k \ WB \Rightarrow \exists (\gamma^i)_{i=1}^{n(k,d)}, (\gamma^i)_{i=1}^{n(k,d)} \text{ connects } (x_i)_{i=1}^k
 \right]=1.
\end{equation}
Indeed clearly, if $(i)$ of Theorem \ref{theresult} holds, so does~\eqref{trewia}.
We prove the other implication by contradiction: if  $(i)$ from Theorem \ref{theresult} is violated, with positive probability
one can find $k$ points in $\mathcal I^u$ that cannot be connected by $n(k,d)$ trajectories in $\omega_u$.
As these points are in $\mathcal I^u$ one can by definition find a sequence $(\gamma^i)^k_{i=1}$ of trajectoris in $\go_u$
such that $x_i\in \gamma^i$ for all $i$. If all the $\go_i$ are distinct, a.s.\ after an eventual
reordering of the sequence we get that
$(x_1,\dots, x_k)$ is well behaved so that \eqref{trewia} cannot hold.
If on the other hand if there are repetition in $(\gamma^i)^k_{i=1}$, one can find one extracts a well behaved subsequence
 $(x'_i)^{k'}_{i=1}$ of $(x_i)^k_{i=1}$
by deleting the points $x_i$ for which
\begin{equation}\label{kwalee}
 \exists j<i, \gamma_i= \gamma_j,
\end{equation}
and reordering the remaining subsequence. Then if \eqref{trewia} holds then one can a.s.\ connect  $(x'_i)^{k'}_{i=1}$ with $n(k',d)$ trajectories.
Then using the definition \eqref{kwalee} one can link all the $(x_i)^k_{i=1}$ together with $n(k',d)+(k-k')\le n(k,d)$ trajectories
(just by using the $\gamma_j$ corresponding to the $k'-k$ remaining points if necessary in addition to the trajectories that connect  $(x'_i)^{k'}_{i=1}$)
which yields a contradiction.
Hence we can focus on proving \eqref{trewia}.

\medskip

Let $\tau_0:=0$ and for $i=1,...,k$ let recursively
\begin{equation}
 \tau_i:=\min\{ s>\tau_{i-1} \ |  \ x_i \in \cI^s \}.
\end{equation}
Note that by definition of $\go_u$, in $\go_{\tau_{i-1},\tau_i}$,
with probability one, there exists a unique trajectory $\gamma^i$ which has $x_i$ in its trace.

Furthermore, by the strong Markov property for Poisson, the law of $\gamma^i$ is independent of that of $\tau_i$ (and the $(\gamma^i)_{i=1}^k$ are independent)
and if we parametrize the oriented trajectory $\gamma^i$ as $(\gamma^i_n)_{n\in \bbZ}$
such that $0$ is the first time $w^i$ visits $x_i$ is $0$, then from the definition of the random interlacement (recall \eqref{e:Qdef})
\begin{equation}
 (\gamma^i_n)_{n\ge 0} \text{ is a simple random walk on } \bbZ^d \text{ started at }x_i.
\end{equation}
Set $\cT:=\max_{i\in[1,k]} \tau_i$.
The event $\{ (x_i)_{i=1}^k \text{ is well behaved}\}$ is equal to $\{ \cT \le u\}$, and up to an event of zero-probability, it coincides with $\{cT< u\}$.

\medskip

Note that conditioned on $\cT$, the process $\go_{\cT,u}$ is independent of $\cT$ and of the $\gamma^i$s.
Hence, setting $X^i:=(\gamma^i_n)_{n\ge 0}$,  we can apply Proposition \ref{p.threecase} and for every $j=1,...,p$ find a sequence of $(d-4)$ trajectories
$(\gamma^i)_{i=k+(j-1)(d-4)+1}^{k+j(d-4)}$ in $\go_{\cT,u}$ that connects together the traces of $X^{2j-1}$, $X^{2j}$, $X^{2j+1}$.

\medskip

One can then conclude by observing that $k+p(d-4)=n(k,d)$ for $k$ odd and that $(\gamma^{i})_{i=1}^{k+p(d-4)}$ is a set of trajectories in $\go_u$ that connects $(x_i)_{i=1}^k$.

\medskip

The case $k=2p$ even is dealt similarly, the only difference being in the last step: we use Proposition \ref{p.threecase} for $i=1,...,p-1$ to connect together
$X^1,\dots,X^{2p-1}$ and Proposition
\ref{p.twocase} to connect $X_{2p-1}$ and $X_{2p}$
with the trajectories $(\gamma^{i})_{i=k+(p-1)(d-4)+1}^{k+(p-1)(d-4)+\lceil d/2 \rceil -2}$ from $\go_{\cT,u}$, and conclude in a similar manner.

\end{proof}

Before the proof of Proposition~\ref{p.threecase} for $d$ odd, (in what follows we always consider that $d$ is odd)
we must introduce additional notation in order to reformulate the statement. Introduce the number
\begin{equation}\label{e.kddef}
k_d:=\lceil d/2 \rceil -2=\frac{d-3}{2}.
\end{equation}



For a finite set $A\subset {\mathbb Z}^d$ and $\sigma\in \tilde{\Omega}$, let $N_A(\sigma)$ be the number of trajectories
in $\sigma$ that intersect $A$. Let $\gamma^1,...,\gamma^{N_A(\sigma)}$ be the trajectories from $\sigma$
that intersect $A$ parameterized so that $\gamma^i_0\in A$ and $\gamma^i_n\notin A$ for all $n<0$
and all $i\in \{1,...,N_A(\sigma)\}$. For $\sigma \in \tilde{\Omega}$, $A\subset {\mathbb Z}^d$ and
$R\in {\mathbb Z}_+$ we define the random set of vertices $\Psi(\sigma,A,R)$ as
\begin{equation}\label{e.psidef}
\Psi(\sigma,A,R):=\bigcup_{i=1}^{N_A(\sigma)}\left(\{\gamma_i(t)\,:\,1\le t\le R^2/8\}\cap B(\gamma_i(0),R/2)\right)
\end{equation}

\begin{definition}\label{d.measdecomp}
Let $r,R \in \bbR_+\cup\{\infty\}$ with $r<R$. For $\sigma\in \tilde{\Omega}$, let $\sigma_R$ be the restriction of $\sigma$
to the trajectories that intersect $B(R)$. Let $\sigma_{r,R}$ be the restriction of $\sigma_R$ to the set of trajectories that do not intersect $B(r)$.
\end{definition}

Observe that $\sigma_r$ and $\sigma_{r,R}$ are supported on disjoint sets of trajectories and that
\begin{equation}\label{d.measdecomp2}
\sigma=\sigma_{r}+\sigma_{r,\infty}.
\end{equation}

Let $(\sigma^{(i,j)})_{1\le i \le 4,\,1\le j \le k_d}$ in $\tilde{\Omega}$ be a family of i.i.d. random interlacement processes with parameter
$\bar u:=u/4k_d$ defined by
\begin{equation}
 \sigma^{(i,j)}:=\go_{\bar u ((i-1)k_d+(j-1)),\bar u ((i-1)k_d+j)},
\end{equation}
and let $(X^i)_{i=1}^3$ be three independent simple random walks starting from $x_1, x_2$ and $x_3$ respectively.
Given $R$, let $T^i(B(R))$ be the first exit time of $X^i$ from $B(R)$ and
$Y_i:=(Y^{i,R}_n)_{n\ge 0}= (X^{i}_{n+T^i(B(R))})_{n\ge 0}$ (and $Y^i=X^i$ when $R=\infty$).
We call $\bP$ the probability measure governing all these processes.
\medskip

We define sequences of random subsets of
${\mathbb Z}^d$. For $0\le r<R \le \infty$, and $i=1,2,3$ set
\begin{equation}\label{e.Adef1}
A^{(1)}_i(r,R)=A^{(1)}_i(R):=\left\{Y^{i,R}_n \,:\,1\le n \le R^2/8\right\}\cap B(Y^i_0,R/2), \ 1\le i \le 3.
\end{equation}
Then recursively for $2\le s \le k_d$ and with $r$, $R$, $j$ as above, define
\begin{equation}\label{e.Adef11}
A^{(j)}_i (r,R):=\Psi\left(\sigma^{(i,j)}_{r,\infty},A^{(j-1)}_i(r,R),R\right)=\Psi\left(\sigma^{(i,j)}_{r,sR},A^{(j-1)}_i(r,R),R\right).
\end{equation}
We simply write $A^{(j)}_i$ when $r=0$ and $R=\infty$.
Note that by construction if $y\in A^{(j)}_i(r,R)$ then there exists a sequence of $k_d-1$ trajectories in $\mathcal I^u$ linking it to the
trace of $X^i$. Thus to prove proposition \ref{p.threecase}, it is in fact sufficient to prove (recall that $2 (k_d-1)+1=d-4$),

\begin{lem}\label{pasouf}
 With probability one, one can find $\gamma\in \sigma^{(4,1)}$ that
connects $A^{(k_d)}_1$, $A^{(k_d)}_2$ and $A^{(1)}_3$ together.
\end{lem}

Inspired by \cite{rath2010connectivity}, we prove Lemma \ref{pasouf} by combining Borel's Lemma and

\begin{lem}\label{l.mainingredient}
Let $d\ge 5$ be odd. Given $x_1,x_2,x_3\in {\mathbb Z}^d$.  Let $R$ and $r$ be integers,
such that $R>\max (|x_1|,|x_2|,|x_3|).$

There exist constants $c(u,d)>0$, $R_0(u,d)< \infty$ and
$\gep(u,d) > 0$, such that for any $r$ and $R$ with $R > R_0$ and $\gep R\ge r^{d-2}$,

\begin{equation}\label{e.bigeq}
\bP\left[\exists \gamma \in \sigma^{(4,1)}_{r,2R}\,:\, \gamma \text{ connects } A^{(k_d)}_1(r,R), A^{(k_d)}_2(r,R), \text{ and } A^{(1)}_3(r,R)    \right]\ge c.
\end{equation}
\end{lem}

We prove Lemma \ref{l.mainingredient} by using a method based on the control of the capacity of the sets $A^{(j)}_i(r,R)$ at the end of the section.

\begin{proof}[Proof of Lemma \ref{pasouf} from Lemma \ref{l.mainingredient}]
For real numbers $r<R$ such that $x_1,x_2,x_3 \in B(R)$, set
\begin{equation}\label{e.Ddefin}
\begin{split}
D(r,R):=\{ \exists \gamma \in \sigma^{(4,1)}\,:\, \gamma \text{ connects } A^{(k_d)}_1(r,R), A^{(k_d)}_2(r,R), \text{ and } A^{(1)}_3(r,R)\}.
\end{split}
\end{equation}

We choose $\epsilon$ so that Lemma~\eqref{l.mainingredient} applies. Let $r_0:=\max(|x_1|,|x_2|,|x_3|)$ and
$R_0:=\epsilon^{-1} r_0^{d-2}$.
For $k\ge 1$, we define recursively
\begin{equation}\label{e.rkdefin}
r_k:=d R_{k-1}^2\text{ and }R_k:= \epsilon^{-1} r_{k}^{d-2}.
\end{equation}
We write $D_k=D_k(X^1,X^2,X^3)$ (reasons for underlining only the dependence in $X^i$ will become clear later)
for $D(r_k,R_k)$ and write $\iota_k$ for
$\ind_{D_k}$, the indicator function of $D_k$.
We want to show that
\begin{equation}
 \bP\left[ D_k \text{ occurs for infinitely many values of } k\right]=1,
\end{equation}
which implies Lemma \ref{pasouf}.

\medskip

We will be done using Borel's Lemma (it is cited as in Lemma 4.12 in \cite{rath2010connectivity})
if we can show that there is some $c$ such that for all $k\ge 1$ we have almost surely
\begin{equation}\label{e.entsD}
{\mathbf P}[D_k|\iota_1,...,\iota_{k-1}]\ge c.
\end{equation}

Let $I_k:=(\iota_i)_{i=1}^k$. It is measurable with respect to the $\sigma$-algebra generated
by the following random objects: $(\{X^i_n\,:\,1\le n \le T_{B(R_{k-1})}+R_{k-1}^2/8\})_{i\le 3}$
and $(\sigma^{(i,j)}_{R_{k-1}(1+k_d)})_{i\le 4, j\le k_d}$.

\medskip

On the other hand,
the event $D_k$ depends on $(\sigma^{(i,j)})_{i\le 4, j\le k_d}$ only through
$(\sigma^{(i,j)}_{r_k,\infty})_{i\le 4, j\le k_d}$. Since $R_{k-1}(1+k_d)<r_k$,
the point measures $(\sigma^{(i,j)}_{R_{k-1}(1+k_d)})_{i\le 4, j\le k_d}$
and $(\sigma^{(i,j)}_{r_k,\infty})_{i\le 4, j\le k_d}$ are independent.

\medskip

Let $\tilde X^i$ be defined by $\tilde X^{i}_n:=X^{i}_{n+B(R_{k-1})+R_{k-1}^2/8}$.
By the strong Markov property, conditionally on $\tilde X^{i}_0$, $\tilde X^{i}$ is
independent of $X^i$ (and its law is the one of a simple random walk).
Furthermore, as $R_{k-1}+R_{k-1}^2/8<r_k$, $D_k$ depends on $X^i$ only through
$\tilde X^i$.

Thus by conditional independence

\begin{equation}\label{e.markovpunch}
{\mathbf P}[D_k\ |\  I_k, (\tilde X^i_0)_{i=1}^3]={\mathbf P}[D_k(\tilde X^1,\tilde X^2, \tilde X^3)]\ | \ (\tilde X^i_0)_{i=1}^3]\ge c,
\end{equation}
where the last inequality follows from Lemma~\ref{l.mainingredient}, with $(x_1,x_2,x_3)$ replaced by
$(\tilde X^i_0)_{i=1}^3$.

\end{proof}

We can now focus on the proof of Lemma~\ref{l.mainingredient}.
Before starting we cite results from \cite{rath2010connectivity} that give estimates on
the capacities of the sets $A^{(j)}_i(r,R)$.

\begin{lem} \cite[Lemmata 4.7, 4.8]{rath2010connectivity} \label{l.capprob}
Let $d\ge 5$ and let $j$ be a positive integer. There exist constants $C_s=C_s(u,d)$ and $\epsilon_s=\epsilon_s(u,d)$
such that for any positive integers $r$ and $R$ with $r^{d-2}\le \epsilon_s R$
and if $x_i\in B(R)$, $i\le 3$ we have
\begin{equation}\label{e.Acapbound}
{\mathbf E}[\cp(A^{(j)}_i(r,R))]\ge C_s R^{\min (d-2, 2s)}.
\end{equation}
Moreover, under the same condition there exist positive finite constants $c_s=c_s(u,d)$,
\begin{equation}\label{Acapboundupper1}
{\mathbf E}[\cp(A^{(s)}_i(r,R))]\le c_s R^{\min (d-2,2s)},
\end{equation}
and
\begin{equation}\label{Acapboundupper2}
{\mathbf E}[\cp(A^{(s)}_i(r,R))^2]\le c_s R^{2\min (d-2,2s)}.
\end{equation}
As a consequence (using Chebychev inequality and changing the value of $c_s$ if needed),
\begin{equation}\label{e.capprob}
{\mathbf P}[\cp(A^{(s)}_i(r,R))\ge c_s R^{\min(d-2,2s)}]\ge c_s.
\end{equation}

\end{lem}

%


\begin{proof}[Proof of Lemma \ref{l.mainingredient}]
We choose the constants $\epsilon_s$ from Lemma~\ref{l.capprob} and assume that $r$ and $R$ are such that Lemma~\ref{l.capprob} applies.
We consider the two following events
\begin{equation}\begin{split}
E_1&:=\{\exists \gamma \in \sigma^{(4,1)}_{2R}\,:\, \gamma \text{ connects } A^{(k_d)}_1(r,R), A^{(k_d)}_2(r,R), \text{ and } A^{(1)}_3(r,R)\},\\
E_2&:=\{\exists \gamma \in \sigma^{(4,1)}_{r}\,:\, \gamma \text{ intersects }  A^{(1)}_3(r,R)\}.
\end{split}\end{equation}
Note that
\begin{equation}\label{toto}
 \{\exists \gamma \in \sigma^{(4,1)}_{r,2R}\,:\, \gamma \text{ connects } A^{(k_d)}_1(r,R), A^{(k_d)}_2(r,R), \text{ and } A^{(1)}_3(r,R)  \}
\supset E_1\setminus E_2.
\end{equation}
Let $\bP^{(4,1)}$ denote the law of $\sigma^{(4,1)}$.
Our main task is to prove that there exists a universal constant $c$ such that
\begin{equation}\label{sashkat}
\bP^{(4,1)}(E_1)
\ge 1-\exp(-c R^{4-2d}\cp(A^{(k_d)}_1(r,R))\cp(A^{(k_d)}_2(r,R))\cp(A^{(1)}_3(r,R))),
\end{equation}
and
\begin{equation} \label{sashkat2}
\bP^{(4,1)}(E_2)
\le u \cp(A^{(1)}_3(r,R))(r/(R-r))^{d-2}.
\end{equation}

According to \eqref{e.capprob} (and independence), choosing $c$ small enough one has with positive probability larger than $c$
\begin{equation}\begin{split}
\cp(A^{(k_d)}_1(r,R))&\ge cR^{2k_d},\\
\cp(A^{(k_d)}_2(r,R))&\ge cR^{2k_d},\\
\cp(A^{(1)}_3(r,R))&\ge  cR^2.
 \end{split}
\end{equation}
Hence \eqref{sashkat}, \eqref{sashkat2} and \eqref{Acapboundupper1} implies (recall that $2k_d=d-3$)
\begin{equation}
 \bP\left[E_1\right]\ge c(1-\exp(-c^4)) \text{ and } \bP\left[E_2\right]\le \bP\left[E_1\right]/2,
\end{equation}
provided that $R$ is large enough. This together with \eqref{toto}, is enough to conclude. From now on, we write $A_1,A_2$ and $A_3$ for $A_1^{(k_d)}(r,R),A_2^{(k_d)}(r,R)$ and $A_3^{(1)}(r,R)$.
In order to prove \eqref{sashkat}  and \eqref{sashkat2} one considers the following construction of $\sigma^{(1,4)}|_{W^{*}_K}$:

\begin{itemize}
 \item Let $\cN$ be a Poisson variable of mean $\bar u \cp(A_3)$.
 \item Conditionally on $\cN$,  let $(\gamma^{i})_{i=1}^{\cN}$ be a sequence of independent (and independent of $\cN$) of
$\cN$ doubly-infinite trajectory with distribution $\pi^*\circ \bar Q_{A_3}$, where
$\bar Q_K(\cdot):=Q_K(\cdot)/Q_K(W_K)$ is the renormalized version of the measure defined in \eqref{e:Qdef}
\end{itemize}
Note that from this construction one has

\begin{equation}\label{rataxes}\begin{split}
 \bP^{(4,1)}[E_1\ | \ \mathcal N]&=1-[1-\bar Q_{A_3}(\gamma \text{ hits } A_1 \text{ and } A_2)]^{\mathcal N},\\
 \bP^{(4,1)}[E_2\ | \ \mathcal N]&=1-[1-\bar Q_{A_3}(\gamma \text{ hits } B(r))]^{\mathcal N},
\end{split}\end{equation}
where $(\gamma_n)_{n\in \bbZ}$ is a trajectory distributed according to $\bar Q_{A_3}$.
Let $P_x$ be the law of the simple random walk $Y$ starting from $x$ and $T_1$ and $T_2$ the hitting times of $A_1$ and $A_2$ respectively.
From the definition of $\bar Q_{A_3}$ we have
\begin{multline}
 \bar Q_{A_3}(\gamma \text{ hits } A_1 \text{ and } A_2)\ge \bar Q_{A_3}(\exists n_2\ge n_1 \ge 0, \gamma_{n_1}\in A_1, \gamma_{n_2}\in A_2)
\\
\ge\min_{x\in A_3} P_x(T_1\le T_2<\infty).
\end{multline}
Moreover using the strong Markov property and the identity
\begin{equation}\label{e.prehitcap}
P_{x}[T_1<\infty]=\sum_{z\in A_1} g(x,z) e_{A_1}(z),
\end{equation}
we get
\begin{multline}\label{e.hitcap}
 P_x(T_1\le T_2<\infty)
\ge \left(\sum_{z\in A_1} g(x,z) e_{A_1}(z)\right)\left(\inf_{y\in A_1} \sum_{z\in A_2} g(y,z) e_{A_2}(z)\right)\\
\ge (\min_{y,z\in B((k_d+1)R)}g(y,z))^2   \left(\sum_{z\in A_1} e_{A_1}\right)  \left(\sum_{z\in A_2} e_{A_2}\right)
\ge cR^{4-2d}\cp(A_1)\cp(A_2),
\end{multline}
(to get the last inequality recall \eqref{capacity} and \eqref{e.hitbounds})
and hence
\begin{equation}
 \bar Q_{A_3}(\gamma \text{ hits } A_1 \text{ and } A_2)\ge cR^{4-2d}\cp(A_1)\cp(A_2).
\end{equation}
Together with the first line of \eqref{rataxes} and averaging with respect to $\cN$, this proves \eqref{sashkat}.
\medskip

Note that $\pi^*\circ \bar Q_{A_3}$ is invariant under change of orientation of the trajectories (see Theorem 1.1 of \cite{sznitmanvacant}) so that if
$\tilde T:=\max\{ n | \gamma_n\in A_3\}$, $(\gamma_n)_{n\ge 0}$ and $(\gamma_{\tilde T-n})_{n\ge 0}$ have the same law.
Hence
\begin{equation}
 \bar Q_{A_3}(\gamma \text{ hits } B(r))\le 2  \bar Q_{A_3}( (\gamma_n)_{n\ge 0} \text{ hits } B(r)).
\end{equation}
Moreover (recall \eqref{e.prehitcap})
\begin{multline}
 \bar Q_{A_3}( (\gamma_n)_{n\ge 0} \text{ hits } B(r))\le \max_{|x|\ge R/2} P_x(H_{B_r}<\infty)
\\
=\max_{|x|\ge R/2}\sum_{z\in A_1} g(x,z) e_{B(r)}(z)\le C(r/(R-r))^{d-2}.
\end{multline}
All of this combined gives
\begin{equation}
  \bar Q_{A_3}(\gamma \text{ hits } B(r))\le C(r/(R-r))^{2-d}.
\end{equation}
Combining with \eqref{rataxes} and averaging with respect to $\cN$ gives
\begin{equation}
 \bP^{(4,1)}[E_2]\le u\cp(A_3)(r/(R-r))^{2-d}.
\end{equation}

\end{proof}

\section{Proof of $(ii)$ in Theorem \ref{theresult}}\label{s.lower}

The aim of this Section is to prove that if one selects $k$ points very distant
from each another in the random interlacement, they are really unlikely to be connected by less than $n(k,d)$ trajectories
(together with a quantitative upper-bound on the probability).

\begin{prop}\label{lowb}
Given $\gep>0$,
 for any $x_1, \ \dots \ , x_k \in \BZ^d$ and for any $n<n(k,d)$ one has
\begin{equation}
 \BP[ (x_i)_{i=1}^k \text{ is $n$-connected }]\le C(d,k,\gep)\max_{i\ne j}|x_i-x_j|^{-1+\gep}.
\end{equation}
\end{prop}

Whereas it is quite intuitive that Proposition \ref{lowb} implies the second half of Theorem \ref{theresult}, the proof is not completely straight-forward
so we write it in full details.

\begin{proof}[Proof of Theorem \ref{theresult} $(ii)$ from Proposition \ref{lowb}]

Set $n<n(k,d)$. For $i=1,\dots,k$ denote
by $B^i_R$ the Euclidean ball of center $R \be_1$ (with $\be_1=(1,0,\dots,0)\in \bbZ^d$) and of radius $R$.
We want to show that the
the probability of the event
\begin{equation}
\mathcal A_R:= \{\exists (x_i)_{i=1}^k\in  \prod_{i=1}^k B^i_R,\ (x_i)_{i=1}^k
\text{ is not $n$-connected }, \forall i\in [1, k], x_i \in \mathcal I_u \}.
\end{equation}
tends to one when $R$ tends to infinity, so that
$\bP\left[\bigcup_{R\ge 1}\mathcal A_R\right]=1$ (which implies  Theorem \ref{theresult} $(ii)$).
According to Proposition \ref{lowb}, using a union bound, one has for $R$ large enough
\begin{equation}
\bP(E^1_R)= \bP\left[\exists (x_i)_{i=1}^k\in  \prod_{i=1}^k B^i_R, (x_i)_{i=1}^k
\text{ is $n$-connected }\right]
\le C R^{kd} e^{−R/2}.
\end{equation}
Moreover from the definition of random interlacement (in particular of the measure $\nu$ in equation \eqref{e:nudef}) we have
\begin{equation}
\bP(E^2_R)=\bP\left[ \forall i\in [1,k], \mathcal I_u \cap B^i_R = \emptyset\right] \le  ke^{-u\cp(B^R_1)}\le  ke^{-cR^{d-2}}.
\end{equation}
Hence we conclude that the probability of $\mathcal A_R= (E_1^R \cup E_2^R)^c$ tends to one as $R\to\infty$.
\end{proof}

We prove the result by induction on $k$.
The strategy that we use is the following: first we encode the way the $k$ points are connected by some tree scheme $\mathcal T$.
This is done in Proposition \ref{lesarbresccool}.
Then we bound from above the probability that $k$ points are connected together using a given scheme by a diagrammatic sum
(Lemma \ref{feyndiag}). Finally we prove an upper-bound on this sum (Proposition \ref{feynman}).
For some tree-schemes the multi-index sum given by Lemma \ref{feyndiag} is infinite and those to be treated separately.
However they are easily dealt with by using the induction hypothesis.

\medskip

\begin{prop}\label{lesarbresccool}
Assume there is a sequence of distinct trajectories $(\gamma^i)^n_{i=1}$ ($\gamma^i\ne \gamma^j$ for $i\ne j$) that connects strictly $(x_i)_{i=1}^k$.

\medskip

Then one can construct:
\begin{itemize}
 \item [(a)] a sequence $(y_i)_{i=1}^m\in (\bbZ^d)^m$, with $m=n+k-1$
and $y_i=x_i$ for $i\le k$,
 \item [(b)]  a
tree $\mathcal T$ with $m$ labeled vertices $A_1, \dots, A_m$, and
$m-1$ oriented edges whose set we call $\CE$,
\item[(c)] a function $t:\CE\to \{1,\dots,n\}$, that to each edge associates a \textsl{type},
\end{itemize}
that satisfies the following properties:
\begin{itemize}
\item[(i)]
The set of oriented edges that share the same label forms an (oriented) path in the tree.
\item[(ii)] For all indices $i\le k$ the edges connected to the vertex $A_i$ (ignoring their orientation)
are all of the same type (hence those vertices have at most degree $2$). For
$i\ge k+1$ the edges connected to the vertex
$A_i$ are of two different types (exactly).
\item[(iii)] If $A_{a_1}A_{a_2}\dots A_{a_l}$, $l\ge 2$ is the path of vertices linked by edges of type  $h$ and  $(\gamma_n^h)_{n\ge0}$
is a time parametrization of $\gamma^h$, then there exists a non-decreasing sequence $b_1, \dots, b_l$ in $\bbZ$ such that $w_{b_i}=y_{a_i}$ for all $i\in [1,l]$.
\end{itemize}

Given $(\mathcal T, \CE, t)$ satisfying $(i)-(ii)$ we say that $(x_i)_{i=1}^k$ is connected with scheme $(\mathcal T, \CE, t)$
(or $\cT$ to simplify notation), if there exists $(y_i)_{i=k+1}^m$ in $(\bbZ^d)^m$,  $(\gamma^i)^n_{i=1}\in (\go_u)^n$, $\gamma^i\ne \gamma^j$ for $i\ne j$,
such that $(iii)$ holds. Furthermore if this holds with $(y_i)_{i=k+1}^m$  fixed, we say that
$(x_i)_{i=1}^k$ is connected with scheme $(\mathcal T, \CE, t)$ using  $(y_i)_{i=k+1}^m$.
\end{prop}

\begin{rem}\rm\label{remkitu}
 Remark that we allow repetition in the sequence $y_1, \dots, y_m$ and that the choice of the tree may not be unique.
Moreover it can easily be checked by the reader that if a sequence of points is connected with scheme $(\mathcal T, \CE, t)$, then the sequence is $n$-connected.
An example for the construction of $\mathcal T$ together with the type function is given in figure \ref{createtre}.
\end{rem}

\begin{figure}[hlt]
\begin{center}
\leavevmode 
\epsfxsize =10 cm
\psfragscanon
\psfrag{w1}{$\gamma_1$}
\psfrag{w2}{$\gamma_2$}
\psfrag{w3}{$\gamma_3$}
\psfrag{w4}{$\gamma_4$}
\psfrag{x1}{$x_1$}
\psfrag{x2}{$x_2$}
\psfrag{x3}{$x_3$}
\psfrag{y1}{$y_4$}
\psfrag{y2}{$y_5$}
\psfrag{y3}{$y_6$}
\psfrag{1}{$1$}
\psfrag{2}{$2$}
\psfrag{3}{$3$}
\psfrag{4}{$4$}
\psfrag{A1}{$A_1$}
\psfrag{A2}{$A_2$}
\psfrag{A3}{$A_3$}
\psfrag{B1}{$A_4$}
\psfrag{B2}{$A_5$}
\psfrag{B3}{$A_6$}
\epsfbox{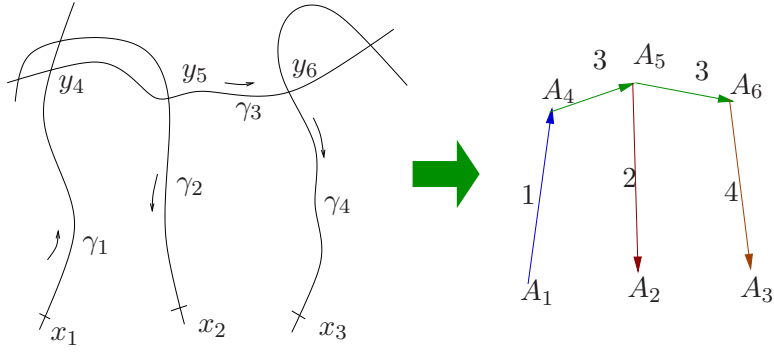}
\end{center}
\caption{\label{createtre}
Examples of the process of tree creation when $k=3$ and $n_1=4$. On the left, the $n_1$ oriented trajectories are represented together with the $x$s and
the points of intersection of the trajectories. On the right this is encoded in the corresponding tree.}
\end{figure}

\begin{proof}
 We prove the statement by induction on $k$.
If $k=2$ and $x_1$ and $x_2$ are strictly connected by $(\gamma^i)_{i=1}^n$,
then it is possible (changing the order of the $w^i$ if necessary) to find $(y_i)_{3\le i\le n+1}$ such that $y_i\in \gamma^{i-2}\cap \gamma^{i-1}$ and
$x_1\in \gamma^1$, $x_2\in \gamma^{n}$.
Then the tree $\mathcal T$ is just the paths $A_1 A_3 A_4 \dots A_{n+1}A_2$ and the edge $A_iA_{i+1}$ has type $i-1$
($A_1A_3$ is of type $1$ and $A_{n+1}A_2$ is of type $n$). Orientation of the edges can then be chosen to satisfy $(iii)$.

\medskip

For $k \ge 3$, we remark that if  $(\gamma_i)_{i=1}^n$ strictly connects $(x_i)_{i=1}^{k}$, then it connects $(x_i)_{i=1}^{k-1}$.
Thus one can find a subsequence of trajectories that strictly connects  $(x_i)_{i=1}^{k-1}$. Hence
after reordering of the indices, one may assume that $(\gamma_i)_{i=1}^{n'}$ for $n'\le n$ strictly connects  $(x_i)_{i=1}^{k-1}$.

\medskip

Using the induction hypothesis one can find a tree $\mathcal T'$ with
$k+n'-2$ vertices $(A_i)_{i\in[1,n'+k-1]\setminus k}$ and a sequence of  $\bbZ^d$ vertices $(y_i)_{i\in [0,n'+k-1]\setminus \{k\}}$,
that satisfies $(i)-(iii)$ (the label $k$ is not used here for a reason that will become clear soon).

\medskip

Assume for the rest of the proof that $n'<n$ (the case $n'=n$ is treated briefly at the end).
Note that since $(\gamma_i)_{i=1}^n$, strictly connects $(x_i)_{i=1}^{k}$, one can find $y_{n'+k}$ in the trace of one of the $(\gamma_i)_{i\le n'}$
(without loss of generality we can assume it belongs to $\tr(\gamma_{n'})$), such that $y_{n'+k}$ and $x_k$ are strictly connected by $(\gamma_i)_{i=n'+1}^k$.

We are now ready to construct the tree $\mathcal T$.
First we construct a path $$ A_{n'+k}A_{n'+k+1}\dots A_{n+k+1}A_{k} $$ composed of $n'-n$ edges of different types ($n'+1$ to $n$), just as one did for the
$k=2$ case.

Then one plugs $A_{n'+k}$ into the old tree $\mathcal T'$ as follows.
Let $A_{a_1}, \dots, A_{a_l}$, $l\ge 2$ be the path of vertices linked by edges of type  $n'$. By $(iii)$ of the induction hypothesis, there exists
a non-decreasing sequence in $\bbZ$, $b_1, \dots, b_l$ such that $\gamma^{n'}_{b_i}=y_{a_i}$ for all $i\in [1,l]$.
By definition $y_{n'+k}=\gamma^{n'}_b$ for some $b\in \bbZ$.

One then constructs $\mathcal T$ from $\CT'$ by adding a new edge of type $n_2$ to include $A_{k+n_2}$ in the tree in the following manner.
\begin{itemize}
\item[(a)] if $b\le b_1$, one adds an edge $A_{k+n'}A_{a_1}$ (and the path $A_{n'+k}A_{n'+k+1}\dots A_{n+k-1}A_k$ previously constructed),
\item[(b)] if $b\in (b_i,b_{i+1}]$ then one replaces the edge $A_{a_i}A_{a_{i+1}}$ by two edges $A_{a_i}A_{n_2+k}$ and $A_{n_2+k}A_{a_{i+1}}$,
\item[(c)] if $b> b_l$ then one adds an edge $A_{a_{l}}A_{n_2+k}$.
\end{itemize}
When $n'=n$ the procedure is exactly the same except that $y_{n'+k}$ is replaced by $x_k$ (and $A_{n'+k}$ by $A_k$) and
that only the second stage is needed (the paths to be plugged is only the single point $A_k$ in this case).
We let the reader check that assumptions $(i)-(iii)$ are satisfied by $\cT$.
\end{proof}

\medskip

According to Proposition \ref{lesarbresccool}, one has
\begin{equation}
 \{(x_i)_{i=1}^k, \text{ is $n$-connected }\}=
\cup_{\CT \in \bbT_n} \{(x_i)_{i=1}^k \text{ is connected with scheme } \mathcal T\},
\end{equation}
where $\bbT_n$ denotes the (finite) set of all schemes $\cT$ with less than $n+k-1$ vertices.
Thus, to prove Proposition \ref{lowb}, we only need to prove that for every $\cT\in \bbT_n$,

\begin{equation}\label{usethescheme}
 \BP\left[ (x_i)_{i=1}^k \text{ is connected with scheme } \mathcal T\right] \le C  \max_{i\ne j}|x_i-x_j|^{-1+\gep}.
\end{equation}

For this purpose we will use the following Lemma that estimates the l.h.s.\ of \eqref{usethescheme}.

\begin{lem}\label{feyndiag}
Let $\CE$ denote the set of edges of $\mathcal T$, a tree with $n+k-1$ vertices. Then
\begin{equation}\label{diagsum}
 \BP\left[ (x_i)_{i=1}^k \text{ is connected with scheme } \mathcal T\right] \le C
\sum_{(y_i)_{i=k+1}^{n+k-1}\in (\bbZ^d)^{n-1}}\prod_{A_iA_j\in \CE}(|y_i-y_j|+1)^{2-d}.
\end{equation}
\end{lem}

\begin{proof}
By a simple union bound it is sufficient to prove that
\begin{equation}\label{croco}
\BP\left[ x_1, \dots, x_k \text{ are connected with scheme } \mathcal T \text{ using } (y_i)_{i=k+1}^m\right]\le
C \prod_{A_iA_j\in \CE}(|y_i-y_j|+1)^{2-d}.
\end{equation}
We prove equation \eqref{croco} in two steps.
First we show that given subsets $E_1,\dots,E_n$ of $W^*$ with finite $\nu$-measure, one has

\begin{equation}\label{croco1}
 \BP[\exists (\gamma^i)_{i=1}^n \in (\go_u)^n,\  \forall i\ne j,\  \gamma^i\ne \gamma^j,\ \forall i,\ \gamma^i \in E_i]\le u^n \prod_{i=1}^n\nu(E_i).
\end{equation}
 Indeed let $\go_{\dd t}=\go_{t,t+\dd t}$ denote infinitesimal division of the Poisson process.
One has
\begin{multline}
  \BP[\exists (\gamma^i)_{i=1}^n \in (\go_u)^n,\  \forall i\ne j,\  \gamma^i\ne \gamma^j,\ \forall i,\  \gamma^i \in E_i]\\
\le \int_{\{ (t_i)_{i=1}^n \in [0,u]^n \ | \ \forall i\ne j \ t_i\ne t_j \}} \BP[\forall i\in [1,n] \  \go_{\dd t_i}\cap E_i\ne \emptyset]\\
= \int_{\{ (t_i)_{i=1}^n \in [0,u]^n \ | \ \forall i\ne j \ t_i\ne t_j \}}\prod_{i=1}^n  (u \nu(E_i)\dd t_i),
\end{multline}
the last equality being obtained using independence.
Secondly we show that for any choice of points $(z_i)_{i=1}^m$ one has
\begin{equation}\label{aswx}
 \nu(\{\gamma\,:\,\gamma \text{ visits } z_1,z_2,\dots,z_m \text{ in that order }\})
\le C_{m} \prod_{i=1}^{m-1}\frac{1}{(|z_{i+1}-z_{i}|+1)^{d-2}}.
\end{equation}
Parameterizing $\gamma=(\gamma_n)_{n\ge 0}$ so that $0$ is the first time of visit of $z_1$ and using the definition of $\nu$ given by
\eqref{e:Qdef}-\eqref{e:nudef} one has
\begin{multline}
  \nu(\{\gamma\,:\,\gamma \text{ visits } z_1,z_2,\dots,z_m \text{ in that order }\})\\
=P_0(\tilde H_0=\infty)P_{z_1}(\exists n_2\le n_3\le \dots\le n_m, \forall i\in [2,m],\ X_{n_i}=y_{n_i})\\
=P_0(\tilde H_0=\infty)\prod_{i=1}^{m-1}P_{z_i}(H_{z_{i+1}}< \infty),
\end{multline}
where the last inequality follows by multiple application of the Markov property at the successive stopping times $H_{z_i}$.
Then \eqref{aswx} is deduced by using \eqref{e.hitbounds}.
\medskip

Combining \eqref{aswx} with \eqref{croco1} used for the events $E_i:=\{ \gamma^i \text{ visits successively } y_{a^i_1}, \dots, y_{a^i_m}\}$ where
$A_{a^i_1}A_{a^i_2}\dots A_{a^i_m}$ are the paths corresponding to oriented edges of type $i$ in $\mathcal T$ we get \eqref{croco}.

\end{proof}

Our problem is that for some schemes in $\bbT_n$, the r.h.s\ of \eqref{diagsum} diverges.
Therefore, we must first identify which are the bad trees for which that happens and prove \eqref{usethescheme} for them without using
\eqref{diagsum}.
Afterwards, we use the following proposition that gives an upper-bound for the r.h.s. of \eqref{diagsum} for the good trees,
 and allow us to conclude.

\begin{prop}\label{feynman}

Given a labeled tree $\mathcal T$ with $k$ leafs $A_1,\dots, A_k$ and $m$ nodes $A_{k+1},\dots, A_{k+m}$ and edges $E$, we associate to each edge a length $l(e)\in [0,d)$.
Suppose that the lengths of the edges are such that:

\begin{itemize}
 \item [(i)] The total length of the tree $l(\mathcal T)=\sum_{e\in E} l(e)$ is strictly smaller than $d(k-1)$.
 \item [(ii)] The length of any (strict) subtree containing at least $k_1$ of the original leafs $A_i$ is at least $d(k_1-1)$.
\end{itemize}
Then for any $\gep>0$ there exists a $C_\gep$ such that, for every  $x_1, \dots, x_k$
\begin{equation}\label{ingsurlesarb}
\sum_{(y_i)_{i=k+1}^{k+m}\in (\bbZ^d)^m}\prod_{A_iA_j\in E}(|y_i-y_j|+1)^{l(A_iA_j)-d}\le C_{\gep}\max_{i\ne j}|x_i-x_j|^{d(k-1)-l(T)+\gep}.
\end{equation}
where we use the convention that $y_i=x_i$ for $i\le k$.
\end{prop}

The proof is postponed to the end of the section.
\medskip

\begin{proof}[Proof of Proposition \ref{lowb}]
The statement is proved by induction on $k$. The case $k=2$ can easily be proved using Proposition \ref{feynman}.
So we only need to focus on the induction step. It is necessary to prove \eqref{usethescheme} for all trees with $k+n-1$ vertices.

\medskip

First consider the trees where there exists $i\le k$ such that $A_i$ is not a leaf (after permutation of the indices we can consider that $A_1$ is not a leaf).
In that case $A_1$ has degree two and the tree $\mathcal T$ can be split into two trees, each of them linking $k_1$ and $k_2$
of the $A_i$s together, and using respectively $n_1$ and $n_2$ types of edges respectively, with $k_1+k_2=k+1$ and $n_1+n_2=n+1$
(recall that the two edges getting out of $A_1$ are of the same type).

\medskip

As $n<n(k,d)$, one has either $n_1<n(k_1,d)$ or $n_2<n(k_2,d)$.
Suppose without loss of generality that $n_1<n(k_1,d)$. In that case a subset of $k_1<k$ vertices is connected by $n_1<n(k_1,d)$ trajectories
(see Remark \ref{remkitu}) and
one can use the induction hypothesis to get \eqref{usethescheme}.
In the rest of the proof we consider only trees for which all the $A_i$, $i\le k$ are leafs.

\medskip

 A connected subgraph of $\CT$ which is a tree and whose leafs are leafs of $\cT$ is said to be a proper subtree of $\mathcal T$.
We consider now the trees $\cT$  with $k+n-1$ vertices that have a proper subtree with $k_1$ vertices and that uses only edges of $n_1$ different types
with $n_1<n(k_1,d)$. Then according to Remark \ref{remkitu}, a subset of $k_1<k$ vertices is connected by $n_1<n(k_1,d)$ trajectories and again one can
prove \eqref{usethescheme} using the induction hypothesis.

\medskip

Now suppose that $\cT$ is a tree for which all subtrees with $k_1<k$ vertices use at least $n(k_1,d)$ type of edges.
To each edge of the tree, we associate an edge-length $2$, and apply Proposition \eqref{feynman} to conclude. Assumption $(i)$ of the proposition
is satisfied since $n<n(k,d)$ and the total number of edges $n+k-2$ is given by Proposition \ref{lesarbresccool}.
Assumption $(ii)$ is satisfied because of our assumption on proper subtrees,
indeed the reader can check that if a proper subtree with $k_1$ vertices uses $n_1$ type of edges, it must have at least $n_1+k_1-2$ edges:
this is because vertices in the tree have degree at most $4$ and that on vertices of degree $3$ two of the incident edges have the same type, and
on vertices of degree $4$, one has two pairs of incident edges with the same type (by $(ii)$ of Proposition \ref{lesarbresccool}).
\end{proof}

\begin{proof}[Proof of Proposition \ref{feynman}]
We perform the proof by induction on $k$.
When $k=2$, it is easy to show that the sum is equal to
$O(|x_1-x_2|^{l(T)-d}(\log |x_1-x_2|)^{\#\{ \text{edges of length $0$} \}})$
where $l(T)$ is the length of the tree.

\medskip

When $k\ge 3$ our strategy is to bound the r.h.s\ of \eqref{ingsurlesarb} by sums corresponding to trees with $k-1$
vertices and
then conclude by using the induction hypothesis.

We remark that if $T$ includes two edges $e$ and $e'$ linked to a common vertex of degree two, one can replace it by a unique
edge of length $l(e')+l(e)+\delta$ (see Figure \ref{treetransform}). Indeed as long as $l(e')+l(e)<d$ we have
\begin{equation}\label{onrajoutdelta}
\sum_{y\in {\mathbb Z}^d} (|x-y|+1)^{l(e)-d} (|y-z|+1)^{l(e')-d}=O( (1-|x-z|)^{l(e)+l(e')+\delta-d}).
\end{equation}
So if one calls $T_1$ the tree obtained after this change (relabeling the vertices of $T_1$ from $A_1,\dots,A_{k+m-1}$, calling $E_1$
the corresponding edge set and for simplicity denote by $l$ the length of the edges on the new tree) one get that there exists a constant $C$ such that

\begin{equation}
\sum_{(y_i)_{i=k+1}^{k+m}\in (\bbZ^d)^m}\prod_{A_iA_j\in E}(|y_i-y_j|+1)^{l(A_iA_j)-d}\le C
\sum_{(y_i)_{i=k+1}^{k+m-1}\in (\bbZ^d)^m}\prod_{A_iA_j\in E_1}(|y_i-y_j|+1)^{l(A_iA_j)-d}.
\end{equation}

Note that adding the $\delta$ is only necessary if one of the edges has length zero in order to avoid having a $\log$ term.
Also note that one can choose the $\delta$ small enough so that after this transformation $l(T_1)\le d(k-1)$.
In particular, this implies that all the edges are still of length smaller than $d$.

\begin{figure}[hlt]
\begin{center}
\leavevmode 
\epsfxsize =10 cm
\psfragscanon
\psfrag{A}{$A$}
\psfrag{B}{$B$}
\psfrag{l1}{$l_1$}
\psfrag{l2}{$l_2$}
\psfrag{l3}{$l_1+l_2+\delta$}
\psfrag{l4}{$l_1+l_2-d$}
\epsfbox{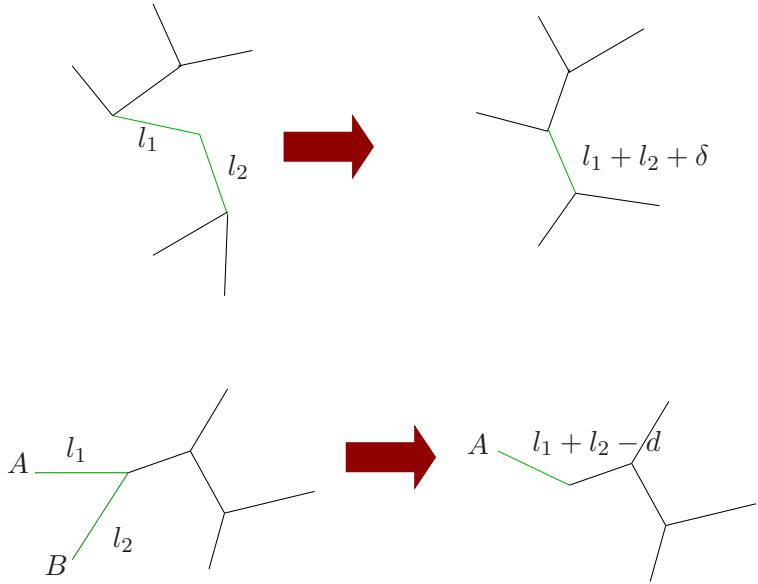}
\end{center}
\caption{\label{treetransform}
Illustration of the two stages of the tree reduction procedure}
\end{figure}

\medskip

Then after having reduced all consecutive edge in this manner we obtain
(what we call the first stage of the reduction) a tree $T'$ with $k+m'$
vertices ($m'\le m$) and $k$ leaves, no vertices
of degree $2$, and satisfying
\begin{equation}
\sum_{(y_i)_{i=k+1}^{k+m}\in (\bbZ^d)^m}\prod_{A_iA_j\in E}(|y_i-y_j|+1)^{l(A_iA_j)-d}\le C
\sum_{(y_i)_{i=k+1}^{k+m'}\in (\bbZ^d)^m}\prod_{A_iA_j\in E'}(|y_i-y_j|+1)^{l(A_iA_j)-d}.
\end{equation}
We can chose the $\delta$ small enough so that $l(T_1)\le l(T)+\gep/2$.
\medskip

After the first stage of the reduction, it is possible to find in $T'$ two leafs at graph distance $2$ of each another
(i.e.\ separated by only two edges):
say without loss of generality that $A_k$ and $A_{k-1}$ are linked to $A_{k+1}$ with edges $A_kA_{k+1}$ and $A_{k+1}A_{k-1}$
 of length $l_1$ resp. $l_2$.
We consider the inequality

\begin{equation}\label{croasfer}
 (x_k-y_{k+1})^{l_1-d} (x_{k-1}-y_{k+1})^{l_2-d}\le (x_k-y_{k+1})^{l_1+l_2-2d}+(x_{k-1}-y_{k+1})^{l_1+l_2-2d}.
\end{equation}

Let $T''_1$ and
$T''_2$ be trees with $k-1$ leafs, obtained by replacing the edges $e$ and $e'$ in $T'$ by a unique edge $e''$ of length $l_1+l_2-d\ge 0$ linking
$A_{k+1}$ and $A_{k}$ resp.\ $A_{k+1}$ and $A_{k-1}$ and deleting the vertex left alone ($A_{k-1}$ resp. $A_k$).
Indeed using \eqref{croasfer} one gets that
\begin{multline}
\sum_{(y_i)_{i=k+1}^{k+m'}\in (\bbZ^d)^{m'}}\prod_{A_iA_j\in E'}(|y_i-y_j|+1)^{l(A_iA_j)-d}\le
\sum_{(y_i)_{i=k+1}^{k+m'}\in (\bbZ^d)^{m'}}\prod_{A_iA_j\in E''_1}(|y_i-y_j|+1)^{l(A_iA_j)-d}\\
+
\sum_{(y_i)_{i=k+1}^{k+m'}\in (\bbZ^d)^{m'}}\prod_{A_iA_j\in E''_1}(|y_i-y_j|+1)^{l(A_iA_j)-d}.
\end{multline}
where $E_1''$ and $E_2''$ denote the edge sets of $T_1''$ and $T_2''$ respectively.

Note that for $i=1,2$, $l(T''_i)=l(T')-d\le l(T)-d+\gep/2$ so that condition $(i)$ is satisfied if $\gep$ is small enough (the new tree has one less leaf).
Note that any proper subtree of $T''$ that does not contain $e''$ is also a proper subtree of $T'$
and any proper subtree $\tau$ of $T''$ that contains $e''$ can  be associated to a subtree $\tau'$ of $T'$ by replacing $e''$ by $e$ and $e'$
(the inverse of the
above transformation)
such that
$l(\tau')=l(\tau)+d$ and $\tau'$ has one more leaf than $\tau$. Hence if condition $(ii)$ is satisfied for $T'$ it is also satisfied for $T''$
so that one can apply the induction hypothesis (with $\gep/2$) on the trees $T''_1$ and $T''_2$  to conclude.

\end{proof}

{\bf Acknowledgments:} The paper is based on a question asked by the first author to
the second author after a seminar given by the second author at Instituto Nacional de Matematica Pura e Aplicada in
Rio de Janeiro. We thank Roberto Imbuzeiro Oliveira for organizing the seminar. JT thanks Vladas Sidoravicius
for hosting his postdoctoral stay at Instituto Nacional de Matematica Pura e Aplicada (IMPA) in Rio de Janeiro.
Both authors acknowledge support of CNPq and hospitality of IMPA.

\bibliography{tykesson}

\def\cprime{$'$} \def\cprime{$'$}
\begin{thebibliography}{BKPS04}

\bibitem[Bela]{Belius1}
David Belius.
\newblock Cover times in the discrete cylinder.
\newblock arXiv:1103.2079.

\bibitem[Belb]{Belius2}
David Belius.
\newblock Gumbel fluctuations for cover times in the discrete torus.
\newblock arXiv:1202.0190.

\bibitem[BKPS04]{benjamini2004geometry}
I.~Benjamini, H.~Kesten, Y.~Peres, and O.~Schramm.
\newblock {Geometry of the uniform spanning forest: Transitions in dimensions
  4, 8, 12,...}
\newblock {\em Annals of mathematics}, 160(2):465--491, 2004.

\bibitem[LL10]{lawler2010random}
G.F. Lawler and V.~Limic.
\newblock {\em {Random walk: a modern introduction}}.
\newblock Cambridge Univ Pr, 2010.

\bibitem[PT11]{PT2011}
Eviatar Procaccia and Johan Tykesson.
\newblock Geometry of the random interlacement.
\newblock {\em Electron. Commun. Probab.}, 16:528--544, 2011.

\bibitem[RS10]{rath2010connectivity}
B.~R{\'a}th and A.~Sapozhnikov.
\newblock {Connectivity properties of random interlacement and intersection of
  random walks}.
\newblock {\em Arxiv preprint arXiv:1012.4711}, 2010.

\bibitem[Szn09a]{Szn09c}
Alain-Sol Sznitman.
\newblock On the domination of random walk on a discrete cylinder by random
  interlacements.
\newblock {\em Electron. J. Probab.}, 14:no. 56, 1670--1704, 2009.

\bibitem[Szn09b]{Szn09b}
Alain-Sol Sznitman.
\newblock Random walks on discrete cylinders and random interlacements.
\newblock {\em Probab. Theory Related Fields}, 145(1-2):143--174, 2009.

\bibitem[Szn09c]{Szn09d}
Alain-Sol Sznitman.
\newblock Upper bound on the disconnection time of discrete cylinders and
  random interlacements.
\newblock {\em Ann. Probab.}, 37(5):1715--1746, 2009.

\bibitem[Szn10]{sznitmanvacant}
Alain-Sol Sznitman.
\newblock Vacant set of random interlacements and percolation.
\newblock {\em Ann. of Math. (2)}, 171(3):2039--2087, 2010.

\bibitem[TW11]{TeixWind10}
Augusto Teixeira and David Windisch.
\newblock On the fragmentation of a torus by random walk.
\newblock {\em Comm. Pure Appl. Math.}, 64(12):1599--1646, 2011.

\bibitem[Win08]{W08}
David Windisch.
\newblock Random walk on a discrete torus and random interlacements.
\newblock {\em Electron. Commun. Probab.}, 13:140--150, 2008.

\end{thebibliography}
\bibliographystyle{alpha}
\end{document}